\documentclass[a4paper,12pt,epsfig]{article}
\usepackage{amsmath,amssymb,amsthm,amscd}
\usepackage{graphics}

\newtheorem{theorem}{Theorem}[section]
\newtheorem{corr}{Corollary}[section]

\newtheorem{prop}{Proposition}[section]
\newtheorem{remark}{Remark}[section]
\newtheorem{df}{Definition}[section]

\newcommand{\conv}{\mathop{\rm conv}\nolimits}

\newcommand{\but}{\mathop{\rm BUT}\nolimits}

\def\text{{\rm }}

\def\stm{{\setminus}}

\def\conv{{\rm conv}}

\def\co-ind{{\operatorname{co-ind}}}

\def\tind{{\operatorname{t-ind\,}}}
\def\hind{{\operatorname{h-ind\,}}}

\def\varnothing{\emptyset}

\let\to=\rightarrow

\title {Borsuk--Ulam Type spaces}

\author {Oleg R. Musin\thanks{The research was carried out at the IITP RAS at the expense of the Russian Foundation for Sciences (project № 14-50-00150)} \; and  Alexey Yu. Volovikov}

\begin{document}
\date{}
\maketitle

\begin{abstract} We consider spaces with free involutions that satisfy the Borsuk--Ulam theorems (BUT--spaces). There are several equivalent definitions for BUT--spaces that can be considered as their properties. Our main technical tool is Yang's cohomological index.
\end{abstract}

\medskip


\section{Introduction}

Throughout this paper the symbol ${\mathbb R}^n$ denotes the Euclidean space of dimension $n$.  We denote by  ${\mathbb S}^n$ the $n$-dimensional sphere. If we consider ${\mathbb S}^n$ as the set of unit vectors $x$ in ${\mathbb R}^{n+1}$, then points $x$ and $-x$ are called {\it antipodal} and the symmetry given by the mapping $A:{\mathbb S}^n \to {\mathbb S}^n$, where
 $A(-x)=x$.

The classical Borsuk--Ulam theorem states:

\medskip

\noindent{\bf Theorem (Borsuk \cite{Borsuk}).} {\it For any continuous mapping $f:{\Bbb S}^n \to {\Bbb R}^n$ there is a point $x\in{\Bbb S}^n$ such that $f(-x)=f(x)$.}

\medskip

One of statements in the Borsuk paper was about coverings of a sphere by closed sets. Note that this result was published by Lusternik and Schnirelmann  in 1930 \cite{LS}, i. e. three years before Borsuk.

\medskip

\noindent{\bf Theorem (Lusternik--Schnirelmann).} {\it ${\Bbb S}^n$ cannot
be covered by $n+1$ closed sets, none containing a pair $(x,-x)$ of diametrically
opposite (antipodal) points.}

\medskip

Since both theorems are true, they are, of course, logically equivalent. But
if their hypotheses are suitably weakened, the resulting statements can be
shown to be equivalent in a more interesting sense. Bacon \cite{Bacon} proved that these two theorems  are equivalent for a very wide class of spaces, see Theorem \ref{tBUTS} in Section 2. Actually, there are a lot of statements that are equivalent to the Borsuk--Ulam theorem \cite{Bacon,KyFan,Mat,Mus, MusSpT,MusK,MusArSp,MusS,MusVo,Ste,ShashkinT,Shashkin99,Tucker,Yang54}.

Let  $X$ be a topological space with a free continuous involution $A:X\to X$.
We say that a pair $(X,A)$ is a {\it BUT (Borsuk-Ulam Type) -- space} if for any continuous mapping  $f:X \to {\Bbb R}^n$ there is a point $x\in X$ such that $f(A(x))=f(x)$.

 Bacon's theorem show that BUT--spaces are also Lusternik--Shnirelman  spaces as well as Tucker and Yang spaces. In Section 2 we consider discrete analogs of the Borsuk--Ulam theorem  \cite{KyFan,Mat,MusSpT,Tucker} and Kakutani type theorem \cite{MusK}. We prove, see Theorems \ref{DBUT} and \ref{KBUT}, that in the case of finite simplicial complexes these theorems are hold for
BUT--spaces.

In Section 3 we consider Yang's cohomological index. It is shown that using this index we can give necessary and sufficient conditions for a space to be BUT, see Theorem \ref{tEQUIV}, Corollaries 3.1 and 3.2.
In Sections 4 and 5 we apply the results of Section 3 for generalizations of several known results for spheres and discs to BUT--spaces and BUT--manifolds.

\section{Equivalent formulations of the Borsuk--Ulam \\ theorem}

We start from the theorem that was proved in this form by Bacon \cite{Bacon}.

\begin{theorem} \label{tBUTS} Let $X$ be a normal topological space with a free continuous involution $A:X\to X$. Then the following statements are equivalent:
\begin{enumerate}

\item $(X,A)$ is a $\but_n$--space (Borsuk--Ulam type space), i. e.,  for any continuous mapping $f:X\to {\Bbb R}^n$ there is $x\in X$ such that $f(A(x)=f(x)$.
\item $(X,A)$ is a $LS_n$--space (Lusternik--Shnirelman  space), i. e.  for
any cover $C_1,\ldots,C_{n+1}$ of  $X$ by $n+1$ closed (respectively,  by  $n+1$ open) sets, there is at least one set containing a pair $(x,A(x))$.
\item $(X,A)$ is a $T_n$--space (Tucker space), i. e. for any covering of $X$ by  a family of $2n$ closed (respectively, of  $2n$ open) sets $\{C_1,C_{-1},\ldots,C_n,C_{-n}\}$, where $C_i$ and $C_{-i}$ are antipodal, i. e., $C_{-i}=A(C_i)$, for all $i=1,\ldots,n$, there is $k$ such that $C_k$ and $C_{-k}$ have a common intersection point.
\item $(X,A)$ is a $TB_n$--space (Tucker--Bacon space), i. e., if each of $C_1,C_2,\ldots, C_{n+2}$ is a closed subset of $X$,
$$
\bigcup\limits_{i=1}^{n+2}{C_i}=X, \quad \bigcup\limits_{i=1}^{n+2}{(C_i\cap A(C_i))}=\emptyset,
$$
and $j\in\{1,2,\ldots,n+1\}$, then there is a point $p$ in $X$ such that
$$
p\in\bigcap\limits_{i=1}^{j}{C_i} \; \mbox{ and } \; A(p)\in\bigcap\limits_{i=j+1}^{n+2}{C_i}.
$$
\item $(X,A)$ is an $Y_n$--space (Yang space).

$Y_n$ can be define  recursively \cite[p. 493]{Bacon}: $Y_0$ contains all $(X,A)$, $(X,A)\in Y_n$ if   a closed subset $F$ in $X$ is such that $F\cup A(F)=X$, then $F\cap A(F)$ is an $Y_{n-1}$--space.

\end{enumerate}
\end{theorem}

\noindent{\bf Remark 1.}  Two facts that ${\Bbb S}^2\in T_2$ and ${\Bbb S}^2\in TB_2$ were first proved by Tucker in 1945 \cite{Tucker}. The class of $Y_n$--spaces was considered by Yang \cite{Yang54}.

\medskip

\noindent{\bf Remark 2.} In our paper \cite{Mus} we also mentioned  the Lusternik--Shnirelman category cat$(Y)$ of a space $Y$,  i. e. the smallest $m$ such that there exists an open covering $U_1,\ldots,U_{m+1}$ of $Y$ with each $U_i$ contractible to a point in $Y$. It is not hard to prove that\\
 {\it If $(X,A)$ is a BUT$_n$-space, then}  cat${(X/A)}\ge$ cat${({\Bbb R}{\Bbb P}^n)}=n$.\\ However, the converse is not true, there are manifolds $(X,A)$ of dimension $n$  with cat$(X/A)=n$ that are not BUT$_n$-spaces.

\medskip
\begin{df} We say that $X$ is a $\but$-space if it is a $\but_n$-space with $n=\dim X$.
\end{df}

Now we consider discrete analogs of the Borsuk--Ulam theorem. We assume that $X$ is a finite simplicial complex and $A:X\to X$ is a free simplicial involution. It is clear, that $X$ is also a topological normal space.

For a triangulation ${\cal T}$ of $X$ denote by $V({\cal T})$ the vertex set of the triangulation ${\cal T}$.
Consider antipodal (equivariant) triangulations ${\cal T}$ of $X$. That means that for any simplex $s$ from ${\cal T}$ its image $A(s)$ is also a simplex in ${\cal T}$
and $s\cap As=\emptyset$. Therefore, $A:V({\cal T})\to V({\cal T})$ is a well defined free involution on the vertex set.

Put $\Pi_m=\{+1,-1,+2,-2,\ldots, +m,-m\}$. Change of sign provides a free involution on $\Pi_m$. A $\Pi_m$-labeling (or $\{\pm 1,\ldots, \pm m\}$-labeling) is a map $L:V({\cal T})\to \Pi_m$. An equivariant (antipodal) labeling $L:V({\cal T})\to \Pi_m$ of an equivariant (antipodal) triangulation ${\cal T}$ of $X$ is a labeling $L$ such that $L(Av)=-L(v)$, $v\in V({\cal T})$.

\begin{theorem} \label{DBUT} Let $X$ be a finite simplicial complex with a free simplicial involution $A:X\to X$. Then the following statements are equivalent:
\begin{enumerate}

\item $(X,A)$ is a $\but_n$--space.

\item  $(X,A)$ is a $\tilde T_n$--space (Tucker space), i. e. for any antipodal triangulation ${\cal T}$ of $X$ and an {antipodal} labeling of the vertices of ${\cal T}$
$$L:V({\cal T})\to \{+1,-1,+2,-2,\ldots, +n,-n\}$$
 there exists a {complementary} edge.

(A labeling $L$ is called {\bf antipodal} if $L(A(v))=-L(v)$ for every vertex $v\in V({\cal T})$.  An edge $[u,v]$ in $T$ is {\bf complementary} if $L(u)=-L(v)$.)

\item  $(X,A)$ is an $F_n$--space (Ky Fan space), i. e. for any antipodal triangulation ${\cal T}$ of $X$ and any antipodal labeling $L:V({\cal T})\to \{+1,-1,+2,-2,\ldots, +m,-m\}$ without   complementary edges  there is an  $n$-simplex in ${\cal T}$ with labels in the form $\{k_0,-k_1,k_2,\ldots,(-1)^nk_n\}$, where  $1\le k_0<k_1<\ldots<k_n\le m$.

\item  $(X,A)$ is a $P_n$--space, i. e.  for any  centrally symmetric set of $2m$ points $\{p_1,p_{-1},\ldots,p_m,p_{-m}\},\, p_{-k}=-p_k,$ in ${\Bbb R}^n$, for any antipodal triangulation ${\cal T}$ of $X$, and for any  antipodal  labeling  $L:V({\cal T})\to \{+1,-1,+2,-2,\ldots, +m,-m\}$ there exists a $k$-simplex $s$ in ${\cal T}$ with labels $\ell_0,\ldots,\ell_ k$ such that the convex hull of points  $\{p_{\ell_0},\ldots,p_{\ell_k}\}\subset{\Bbb R}^n$ contains the origin $0$ of ${\Bbb R}^n$.

\end{enumerate}
\end{theorem}
\begin{proof} The fact that ${\Bbb S}^2\in \tilde T_2$ was first proved by Tucker \cite{Tucker}.  Using the same arguments as in \cite[Theorem 2.3.2]{Mat}, it can be proved that statements 1 and 2 are equivalent. It is also follows from the equivalence of statements 1 and 3 in Theorem \ref{tBUTS}.

The equivalence of statements 1 and 3 directly follows from a proof of the Ky Fan lemma in \cite{KyFan}, see also \cite[Theorem 5.3]{MusSpT}.

In fact, in our paper \cite{MusSpT} it is proved that if $(X,A)\in P_n$ then $(X,A)\in \tilde T_n$ and $(X,A)\in F_n$. Namely, to prove that $(X,A)\in \tilde T_n$ we can take as points $p_i$ vertices of a regular cross polytope in ${\Bbb R}^n$, and for $(X,A)\in F_n$ these points can be chosen as vertices of ASC $(m,n)$ polytope (see, \cite[Theorem 5.2]{MusSpT}). The equivalence of statements 1 and 4 follows from \cite[Theorem 4.2]{MusSpT}.
\end{proof}

For a set $S$ denote by $2^S$ the set of all subsets of $S$, including the empty set and $S$ itself. Let $X$ and $Y$ be topological spaces. We will say that $F:X\to 2^Y$ is a {\it set-valued function on $X$ with a closed graph} if graph$(F)=\{(x,y)\in X\times Y: y\in F(x)\}$ is a closed as a subset of $X\times Y$. Such a set-valued mapping is a u.s.c. mapping (upper semi continuous mapping).

\begin{theorem} \label{KBUT} Let $X$ be a compact space
with a free involution $A:X\to X$. Then $(X,A)$ is a $\but_n$--space if and only if for any set--valued function $F:X\to 2^{{\Bbb R}^n}$  with a closed graph  such that for all $x\in X$ the set $F(x)$ is non-empty compact convex subspace of ${\Bbb R}^n$ and  contains $y$ with $(-y)\in F(A(x))$, there is $x_0\in X$ such that  $F(x_0)$ covers the origin $0\in{\Bbb R}^n$.
\end{theorem}
\begin{proof} First we give the proof for a finite simplicial complex
with a free simplicial involution.
Let $\{{\cal T}_i,\, i=1,2,\ldots\}$ be a sequence of antipodal triangulations of $X$, where ${\cal T}_{i+1}$ is the barycentric triangulation of ${\cal T}_i$.  Now we define an antipodal mapping $f_i:X\to {\Bbb R}^n$.  Note that if we define $f_i$ for all vertices $V({\cal T}_i)$, then $f_i$ can be piece-wise linearly extended to ${\cal T}_i$, i. e.  to $f_i:X\to {\Bbb R}^n$.
	
	 Let $x\in V({\cal T}_i)$. Then $A(x)$ is also a vertex of ${\cal T}_i$. Let $y\in F(x)$ be a point in ${\Bbb R}^d$ such that $(-y)\in F(A(x))$. Then for the pair $(x,A(x))$ set
$f_i(x):=y$ and $f_i(A(x)):=-y$.

Now we have a continuous antipodal mapping $f_i:X\to {\Bbb R}^n$. By assumption, $(X,A)\in\but_n$, therefore, there is a point $x_i\in X$ such that $f_i(x_i)=0$. Now, suppose that $x_i$ lies in a simplex $s_i$ of ${\cal T}_i$ with vertices $v_0^i,\dots,v_d^i$  and let $y_k^i:=f_i(v_k^i)$. (Note that $y_k^i\in F(v_k^i)\subset{\Bbb R}^n$.) We have $0\in \conv(y_0^i,\dots,y_d^i)$ in ${\Bbb R}^n$. Actually, we may assume that $d=n$. Indeed, if $d<n$, then we set $y_{d+i}=y_d,\, i=1,\ldots,n-d$. In the case when $d>n$ there exists a subset of $n+1$ points (without loss of generality one may assume that it is $\{y_0,\ldots,y_n\}$) such that its convex hull contains 0.

Let $t_0^i,\dots,t_n^i$ be the barycentric coordinates of $0$ relative to the simplex  in ${\Bbb R}^n$ with vertices $y_0^i,\dots,y_n^i$. Then
$$
\sum\limits_{k=0}^n{t_k^iy_k^i}=0.
$$

From compactness assumptions the sequences $\{x_i\}_{i=1}^\infty,\, \{t_k^i\}_{i=1}^\infty,\, \{y_k^i\}_{i=1}^\infty$, $k=0,1,\ldots,n$, may, after possibly renumbering them, be assumed to converge to points $x_0, t_k$ and $y_k, \, k=0,1,\ldots,n$ respectively.

Finally, since the graph of $F$ is closed, we have $y_k\in F(x_0),\, k=0,1,\ldots,n$. Since
$F$ takes convex values, $F(x_0)$ is convex and so  $0\in{\Bbb R}^n$, being a convex combination of the $y_k\in F(x_0)$. Thus $x_0$ is the required point in $X$.

In general consider the nerve $|{\cal V}|$ of a finite equivariant covering of $X$ Note that ${\cal V}$ is a finite simplicial complex with a free simplicial involution. For each $U\in {\cal V}$ take $x_U\in U$ in such a way that the collection of points $\{x_U\}$ is invariant under the action of the involution
$A:X\to X$. Define a set valued mapping on the vertex set of $|{\cal V}|$ as follows. A vertex of $|{\cal V}|$ corresponds to an element $U\in {\cal V}$ and
we define its image to be $F(x_U)$. By the above procedure we obtain a map $f_{{\cal V}}:|{\cal V}|\to \mathbb R^n$ which is affine on each simplex of ${\cal V}$ and antipodal with respect to the involution on $|{c}|$. Since a canonical map $X\to |{\cal V}|$ is equivariant, $|{\cal V}|$ is a $\but_n$-space, hence
$f_{{\cal V}}$ covers $0\in \mathbb R^n$. Using star refinement argument we see that for any finite covering $\cal U$ of $X$ there exists $U\in {\cal U}$ and points $x_i({\cal U})\in U$, $y_i({\cal U})\in F(x_i({\cal U}))$, $i=1,\dots,n$ such that $0\in\conv (y_1({\cal U}),\dots,y_n({\cal U}))$. Hence there exist $x_0\in X$ and $y_1,\dots,y_n\in F(x_0)$ such that $0\in\conv (y_1,\dots,y_n)$ and since $F(x_0)$ is convex, we obtain $0\in F(x_0)$.
\end{proof}

\section{Indexes and BUT--spaces}

In this section we recall properties of topological and homological indexes introduced by Yang~\cite{Yang54}, Schwarz~\cite{Sv} and Conner--Floyd~\cite{CF60}.

\subsection{Topological index}

The topological index for free $\mathbb Z_2$-spaces was invented by C.T.C.\,Yang in 1954 under the name $B$-index.

\begin{df}
Let $X$ be a free $\mathbb Z_2$-space. Denote by $\tind X$ the minimum $n$ such that there exists a continuos antipodal mapping
$X\to {\Bbb S}^n$. If no such $n$ exists, then $\tind X=\infty$.
\end{df}

Note that $X$ is a BUT$_n$-space if and only if $\tind X\ge n$.

\subsection{Yang's (co)homological index}

Yang \cite{Yang54} used homology groups in his definition. For us is more convenient to use \v Cech cohomology.

In what follows we denote by $T$ a free involution on a paracompact space $X$ and omit coefficients $\mathbb Z_2$ in cohomology groups.

\medskip

{\bf 1. First definition.}

\medskip

To define $\hind\!_{2}(\,\cdot\,)$ consider Smith sequence
$$
\begin{CD}
\dots @>>>H^k(X/T) @>\pi^*>> H^k(X) @>\pi^{!}>> H^k(X/T) @>\delta>> H^{k+1}(X/T)@>>>\dots
\end{CD}
$$

Put
$$
\delta^n=\delta\circ\delta\circ\dots\circ\delta: H^0(X/T)\to H^n(X/T).
$$

Index $\hind\!_{2}(X)$ equals the greatest $n$ such that $\delta^n(1)\not=0$.

\bigskip

{\bf 2. Second definition.}

\medskip

Put $w(X)=\delta(1)$.

It can be shown that $w(X)$ coincides with the first Stiefel--Whitney class of the line bundle $(X\times \mathbb R)/T\to X/T$, and that $w^n(X)=\delta^n(1)$.

By defininition $\hind\!_{2}(X)$ is the maximal $n$ such that $w^n(X)\not=0$.

\bigskip

{\bf 3. Third definition.}

\medskip

Consider $E_G\to B_G$ for $G=\mathbb Z_2$, i.e. $S^{\infty}\to \mathbb RP^{\infty}$.

\bigskip

The covering $X\to X/T$ is induced from the universal covering $S^{\infty}\to \mathbb RP^{\infty}$ by some map $\mu:X\to \mathbb RP^{\infty}$ (unique up to a homotopy), and $H^*(\mathbb RP^{\infty};\mathbb  Z_2)=\mathbb Z_2[w]$, $w\in H^1(\mathbb RP^{\infty};\mathbb  Z_2)$.

$\hind\!_{2}(X)$ equals maximal $n$ such that $\mu:H^i(\mathbb RP^{\infty};\mathbb  Z_2)\to H^i(X/T;\mathbb  Z_2)$ is a monomorphism for all $i\le n$,
or equivalently equals maximal $n$ such that $\mu:H^n(\mathbb RP^{\infty};\mathbb  Z_2)\to H^n(X/T;\mathbb  Z_2)$ is a monomorphism.

\subsection{Properties of $\hind\!_{2}(\,\cdot\,)$}

${}$

1. If there exists a $\mathbb Z_2$-map $X\to Y$ then \, $\hind\!_{2}(X)\leq \hind\!_{2}(Y)$.

2. If $X=A \bigcup B$ are open invariant subspaces, then $$\hind\!_{2}(X)\le \hind\!_{2}(A)+\hind\!_{2}(B)+1.$$

3. Tautness: If Y is a closed invariant subspace of $X$, then there exists an open invariant neighborhood of $Y$ such that
$\hind\!_{2}(Y)=\hind\!_{2}(U)$.

4. $\hind\!_{2}(X)>0$ if $X$ is connected.

5. If $X$ is finite dimensional or compact then $\hind\!_{2}(X)<\infty$.

6. Assume that $X$ is connected and $H^i(X;{\mathbb Z_2})=0$ for $0<i<N$. Then $\hind\!_{2}(X)\ge N$.

7. Assume that $X$ is finite dimensional and $H^i(X;{\mathbb Z_2})=0$ for $i>d$. Then $\hind\!_{2}(X)\le d$.

8. If there exists an equivariant map $f:X\to Y$ and $\hind\!_{2}(X)= \hind\!_{2}(Y)=k<\infty$ then $0\not=f^*:H^k(Y;\mathbb Z_2)\to H^k(X;\mathbb Z_2)$.

\medskip

In particular $\hind\!_{2}(S^n)=n$ and index is stable, that is, $\hind\!_{2}(X*\mathbb Z_2)=\hind\!_{2}(X)+1$.

\subsection{Proof of Property 8 for $\hind\!_{2}(\,\cdot\,)$}

This property was first proved by Conner and Floyd \cite{CF60}. Here we present more simple proof.

Denote by $\pi:X\to X/T$ the  projection and put $k=\hind\!_{2}X=\hind\!_{2}Y$. An equivariant map $f:X\to Y$ induces
a map of factor spaces $X/T\to Y/T$ and we have a commutative diagram
$$
\begin{CD}
H^k(X/T) @>\pi^*>> H^k(X) @>\pi^{!}>> H^k(X/T) @>\delta>> H^{k+1}(X/T)\\
@AAA @AAf^*A @AAA @AAA \\
H^k(Y/T) @>>> H^k(Y) @>>> H^k(Y/T) @>>\delta> H^{k+1}(Y/T) \end{CD}
$$
coefficients $\mathbb Z_2$ are omitted.

Since $\delta w^k(Y)=0$, there exists $\alpha\in H^k(Y)$ which is mapped to $w^k(Y)$.
Now $w^k(Y)$ is mapped to $w^k(X)$ and it follows that $f^*\alpha\not=0$, otherwise $w^k(X)=0$.

The following theorem is a partial converse to  Property~8 (see also \cite[Proposition 3.3]{Vo00}).

\begin{theorem}\label{tEQUIV}
Let $X$ and $Y$ be free $\mathbb Z_2$-spaces and $f: X\to Y$  an equivariant map.
Assume that

a) $\hind\!_{2}(Y) = k$,

b) $\dim X=k$,

c) $H^k(X)=H^k(Y)=\mathbb Z_2$,

d)  $f^*:H^k(Y)\to H^k(X)$ is an isomorphism.

Then \, $\hind\!_{2}(X) = k$.
\end{theorem}
\begin{proof} Consider the above diagram.
Since $w^k(Y)\not=0$ and $\delta w^k(Y)=0$, the generator of $H^k(Y)=\mathbb Z_2$ is mapped onto $w^k(Y)$.
Now $w^k(Y)$ is mapped to $w^k(X)$ and from assumption d) it follows that the generator of $H^k(X)=\mathbb Z_2$
is mapped onto $w^k(X)$. Now we argue by contradiction. If $w^k(X)=0$ then $\pi^{!}:H^k(X) \to H^k(X/T)$ is trivial. Hence
$\delta:H^k(X/T) \to H^{k+1}(X/T)$ is a monomorphism and since $H^{k+1}(X/T)=0$, we have $H^k(X/T)=0$. From the exactness of
the upper row we obtain $H^k(X)=0$, a contradiction.
\end{proof}

\subsection{Cohomological index for integer coefficients}

Defined similarly

\begin{itemize}

\item via Smith--Richardson sequences,

\item via Stiefel--Whitney class $u(X)\in H^1(X/T;\mathbb Z_t)$,

\item via cohomology of $\mathbb RP^{\infty}$ with coefficients in $\mathbb Z$ and $\mathbb Z_t$.

\end{itemize}

\bigskip

Here $\mathbb Z_t$ is $\mathbb Z$ with antipodal action of $T$:\,\,\, $Tm=-m,\,\,\, m\in \mathbb Z$.

Note that $\mathbb Z_t\otimes\mathbb Z_t=\mathbb Z$, hence $u^{2k}(X)\in H^{2k}(X/T;\mathbb Z)$
and $u^{2k+1}(X)\in H^{2k+1}(X/T;\mathbb Z_t)$.

Moreover,  $u^k(X)$ is mapped to $w^k(X)$ under the homomorphisms $H^{*}(X/T;\mathbb Z)\to H^{*}(X/T;\mathbb Z_2)$
and $H^{*}(X/T;\mathbb Z_t)\to H^{*}(X/T;\mathbb Z_2)$ induced by coefficient
homomorphisms $\mathbb Z\to \mathbb Z_2$, $\mathbb Z_t\to \mathbb Z_2$.

One has $H^{2k}(\mathbb RP^{\infty};\mathbb Z)=\mathbb Z_2$ for $k>0$, $H^{2k+1}(\mathbb RP^{\infty};\mathbb Z_t)=\mathbb Z_2$

The class $u^k(X)$ is the first obstruction class to the existence of an equivariant map $X\to S^{k-1}$.

The classes $u^k(Y\times Y\setminus \Delta)$ are van Kampen obstruction classes \cite{Mel}. Here $\Delta$ is the diagonal in $Y\times Y$
and free involution $T:Y\times Y\to Y\times Y$ changes factors: $T(y_1,y_2)=(y_2,y_1)$.

\bigskip

The index $\hind_{\mathbb Z}(\,\cdot\,)$ has properties similar to those of  $\hind\!_2(\,\cdot\,)$:

\medskip

1. If there exists an equivariant map $X\to Y$ then $\hind_{\mathbb Z}(X)\leq \hind_{\mathbb Z}(Y)$.

2. If $X=A \bigcup B$ are open invariant subspaces, then $$\hind_{\mathbb Z}(X)\le \hind_{\mathbb Z}(A)+\hind_{\mathbb Z}(B)+1.$$

3. Tautness: If Y is a closed invariant subspace of $X$, then there exists an open invariant neighborhood of $Y$ such that
$\hind_{\mathbb Z}(Y)=\hind_{\mathbb Z}(U)$.

4. $\hind_{\mathbb Z}(X)>0$ if $X$ is connected.

5. If $X$ is finite dimensional or compact then $\hind_{\mathbb Z}(X)<\infty$.

6. Assume that $X$ is connected and $H^i(X;{\mathbb Z})=0$ for $0<i<N$. Then $\hind_{\mathbb Z}(X)\ge N$.

7. Assume that $X$ is finite dimensional and $H^i(X;{\mathbb Z})=0$ for $i>d$. Then $\hind_{\mathbb Z}(X)\le d$.

8. If there exists an equivariant map $f:X\to Y$ and $\hind_{\mathbb Z}(X)= \hind_{\mathbb Z}(Y)=k<\infty$ then $0\not=f^*:H^k(Y;\mathbb Z)\to H^k(X;\mathbb Z)$.

\bigskip

To prove Property 8 for $\hind_{\mathbb Z}(\,\cdot\,)$ consider Richardson--Smith sequences
$$
\begin{CD}
\dots\to H^k(X/T;\mathbb Z_t) \to H^k(X;\mathbb Z) \to H^k(X/T;\mathbb Z) @>\delta_1>> H^{k+1}(X/T;\mathbb Z_t)\to\dots\\
\end{CD}
$$
and
$$
\begin{CD}
\dots\to H^k(X/T;\mathbb Z) \to H^k(X;\mathbb Z) \to H^k(X/T;\mathbb Z_t) @>\delta_2>> H^{k+1}(X/T;\mathbb Z)\to\dots\\
\end{CD}
$$
which are are cohomological sequences associated to exact triples of coefficient sheaves on $X/T$:
$$
0\to\mathbb Z_t\to{\cal A}\to\mathbb Z\to 0 \mbox{\quad\,\,\,and\,\,\,\quad} 0\to\mathbb Z\to {\cal A}\to\mathbb Z_t\to 0,
$$ respectively, $\cal A$
is the direct image of the constant sheaf $\mathbb Z$ on $X$.

We put
$$
\begin{aligned}
&\Delta_{2n}=(\delta_2\circ\delta_1)^n=\delta_2\circ\delta_1\circ\dots\circ\delta_2\circ\delta_1\colon H^0(X/T;\mathbb Z)\to H^{2n}(X/T;\mathbb Z)\\
&\mbox{and}\hphantom{xxxxxxxxxxxxxxxxxxxxxxxxxxxxxxx}\\
&\Delta_{2n+1}=\delta_1\circ (\delta_2\circ\delta_1)^n=\delta_1\circ\delta_2\circ\dots\circ\delta_2\circ\delta_1\colon
H^0(X/T;\mathbb Z)\to H^{2n+1}(X/T;\mathbb Z_t).
\end{aligned}
$$
Then $u^m(X)=\Delta_m(1)$ and index $\hind_{\mathbb Z}(X)$ equals the greatest $m$ such that $u^m(X)\not=0$.

Now assume that $f\colon X\to Y$ is equivariant and $\hind_{\mathbb Z}(X)=\hind_{\mathbb Z}(Y)=k$.

We have $u^k(Y)\not=0$ and $\delta_1 u^k(Y)=0$ if $k$ is even and $\delta_2 u^k(Y)=0$ if $k$ is odd.
Using first commutative diagram for $k$ even
$$
\begin{CD}
H^k(X/T;\mathbb Z_t) @>>> H^k(X;\mathbb Z) @>>> H^k(X/T;\mathbb Z) @>\delta_1>> H^{k+1}(X/T;\mathbb Z_t)\\
@AAA @AAf^*A @AAA @AAA \\
H^k(Y/T;\mathbb Z_t)@>>> H^k(Y;\mathbb Z) @>>> H^k(Y/T;\mathbb Z) @>>\delta_1> H^{k+1}(Y/T;\mathbb Z_t) \end{CD}
$$
and second for odd $k$
$$
\begin{CD}
H^k(X/T;\mathbb Z) @>>> H^k(X;\mathbb Z) @>>> H^k(X/T;\mathbb Z_t) @>\delta_2>> H^{k+1}(X/T;\mathbb Z)\\
@AAA @AAf^*A @AAA @AAA \\
H^k(Y/T;\mathbb Z) @>>> H^k(Y;\mathbb Z) @>>> H^k(Y/T;\mathbb Z_t) @>>\delta_2> H^{k+1}(Y/T;\mathbb Z) \end{CD}
$$
we see that there exists $\alpha\in H^k(Y;\mathbb Z)$ which is mapped to $u^k(Y)$.
Now $u^k(Y)$ is mapped to $u^k(X)$ and it follows that $f^*\alpha\not=0$ is mapped to $u^k(X)$.
Hence $f^*\alpha\not=0$, otherwise $u^k(X)=0$.

\bigskip
We have $\hind\!_2 X\le \hind_{\mathbb Z} X\le \tind X$. Note also that $\tind (\,\cdot\,)$ possesses properties similar to properties 1 -- 6
of cohomological index.

If $\dim X=\tind X$ then $\dim X=\tind X=\hind_{\mathbb Z} X$.
Yang \cite{Yang54} showed that there exists finite complex $X$ such that
$\hind_2 X< \hind_{\mathbb Z} X= \tind X=\dim X$. 

Reducing coefficients modulo 2 ($\mathbb Z_t\to \mathbb Z_2$ and $\mathbb Z\to \mathbb Z_2$)
turns both Richardson-Smith sequences into Smith sequence and the homomorphism induced by the reducing maps $u(X)$ onto $w(X)$.
In even dimension the coefficients are constant. Since $u(X)$ has order 2, this class belongs to 2-primary subgroup of $H^n(X/T;\mathbb Z)$,
where $n=\dim X$. Assume that 2-primary component of $H^n(X;\mathbb Z)$ is an elementary 2-group (i.e. a direct sum of $\mathbb Z_2$) then
the restriction of the homomorphism $H^n(X/T;\mathbb Z)\to H^n(X/T;\mathbb Z_2)$ on the 2-primary component is a monomorphism. Hence if
$u^n(X)\not=0$ then $w^n(X)\not=0$. For $n$ odd we need to replace constant coefficients $\mathbb Z$ by $\mathbb Z_t$. However we can reformulate
this sufficient condition for $n$ odd. Since both indicies are stable, we can replace $X$ by its suspension $SX=X*\mathbb Z_2$. Thus we have
the following sufficient condition:

\begin{prop} Let $X$ be a finite simplicial complex with a free simplicial action. Assume that $\tind X=\dim X=n$. If $n$ is even assume that
2-primary component of $H^n(X/T;\mathbb Z)$ is an elementary 2-group. If $n$ is odd assume that
2-primary component of $H^{n+1}(SX/T;\mathbb Z)$ is an elementary 2-group.

Then $\hind_2 X=\dim X=n$.
\end{prop}

This statement holds also for spaces equivariantly homotopic to finite simplicial (or $CW$-) complexes and compact $ANR$-spaces and hence can be applied to manifolds. Since higher dimensional integer cohomological group of a closed connected manifold is either $\mathbb Z$ or $\mathbb Z_2$, we obtain

\begin{corr} Let $M$ be a closed topological manifold with a free involution. Then \, $\tind M=\dim M$ if and only if \, $\hind_2 M=\dim M$.
\end{corr}

\begin{df} An $n$-dimensional pseudomanifold $M$
is a finite simplicial complex such that

a) $M$ is a finite union of $n$-dimensional simplices and such that
each $(n-1)$-dimensional simplex is a face of precisely two (respectively one or two) $n$-dimensional simplices.

b) $M$ is strongly connected, i.e. any two $n$-dimensional simplices can be joined by a chain of
$n$-dimensional simplices in which each pair of neighboring simplices have a common $(n-1)$-dimensional face.

A complex satisfying condition a) is called an almost pseudomanifold.

The boundary $\partial M$ of $M$ consists of $(n-1)$-dimensional simplices each of which is
a face of precisely one $n$-dimensional simplex.

A pseudomanifold (almost pseudomanifold) is called closed if $\partial M=\emptyset$.
\end{df}

It follows from the definition that an almost pseudomanifold is a compact space and a pseudomanifold is a compact connected space.
An almost pseudomanifold $M$ is the finite union of pseudomanifolds $M=M_1\cup\dots\cup M_k$ where $\dim M_i\cap M_j\le \dim M-2$;
$i\not=j$. A boundary of $n$-dimensional pseudomanifold is not an almost pseudomanifold in general.

\begin{corr} Let $M$ be a closed almost pseudomanifold with a free simplicial involution. Then \, $\tind M=\dim M$ if and only if \, $\hind_2 M=\dim M$.
\end{corr}
\begin{proof}
If $M$ is a closed almost pseudomanifold with a free simplicial involution, then $M/T$, $SM$ and $SM/T$ \, are closed almost pseudomanifolds also.
Since the higher dimensional integer cohomology group of a closed almost pseudomanifold is a sum of a free
Abelian group and an elementary Abelian 2-group (for a pseudomanifold it is either $\mathbb Z$ or $\mathbb Z_2$), we are done.
\end{proof}


\section{Bounded BUT--spaces}

The main tool in this section is  Property 8 of cohomological indicies.

Let $X$ be a subspace of $Z$ (no action of $\mathbb Z_2$ on $Z\setminus X$ assumed, and we don't assume here that $Z$ is
a simplicial complex). Denote by $i:X\to Z$ the inclusion.

\begin{prop} Assume that $\hind\!_2 X\geq d-1$ and $f\colon Z\to \mathbb R^d$ is a map which is equivariant on X
(i.e. $f(Tx)=-f(x)$ for any $x\in X$).

If $0=i^*:H^{d-1}(Z;\mathbb Z_2)\to H^{d-1}(X;\mathbb Z_2)$ then $f^{-1}(0)\not=\varnothing$.
\end{prop}
\begin{proof} Suppose that $f^{-1}(0)=\varnothing$. Then $f:Z\to\mathbb R^d\setminus \{0\}$, and it is equivariant on $X$.
Since $\hind\!_2 X=\hind\!_2 (\mathbb R^d\setminus \{0\})=d-1$ we have a contradiction with property~8.
\end{proof}
Similarly we can use $\hind_{\mathbb Z}(\,\cdot\,)$, for example we have:
\begin{prop} Assume that $\hind_{\mathbb Z} X\geq d-1$ and $f\colon Z\to \mathbb R^d$ is a map which is equivariant on X
(i.e. $f(Tx)=-f(x)$ for any $x\in X$).

If $0=i^*:H^{d-1}(Z;\mathbb Z)\to H^{d-1}(X;\mathbb Z)$ then $f^{-1}(0)\not=\varnothing$.
\end{prop}

In the simplicial case ($X$ is a subcomplex of finite simplicial complex $Z$) we obtain:
\begin{corr}
Assume that\, $\dim X=d-1=\tind X$. Assume that the inclusion $X\subset Z$ induces trivial homomorphism in
$d$-dimensional integer homology.
Then for any triangulation of $Z$ which is equivariant on
$X$ and for any $\Pi_d$-labeling of $Z$ which is equivariant on
$X$ there exists a complementary edge.
\end{corr}

\begin{theorem} Let $X$ be a free $\mathbb Z_2$-space, $i:X\subset Z$. Let $K$ be a free $\mathbb Z_2$-space with involution  $A:K\to K$ and $f:Z\to K$ is a map equivariant on $X$.
Assume that $\hind\!_2 X=d-1$ and that the inclusion $i:X\subset Z$ induces trivial homomorphism $0=i^*:H^{d-1}(Z;\mathbb Z_2)\to H^{d-1}(X;\mathbb Z_2)$.
Then $\hind\!_2 K \ge d$.

If in addition $K$ is a connected closed $d$-dimensional topological manifold or a pseudomanifold then for any $y\in K$ at least one of the sets $f^{-1}(y)$, $f^{-1}(Ay)$ is nonempty.
\end{theorem}
\begin{proof} The map $f\circ i:X\to K$ is equivariant, so $\hind\!_2 K \ge \hind\!_2 X= d-1$. Since $(f\circ i)^*=i^*\circ f^*=0$ in dimension $d-1$, it follows from property~8 of index that $\hind\!_2 K \not= d-1$. Therefore $\hind\!_2 K \ge d$.

When $K$ is a manifold we argue by contradiction. Let $y\in K$ be a point such that $f^{-1}(y)=\emptyset=f^{-1}(Ay)$. Then $f$ maps $Z$ to
$K\setminus\{y\cup Ay\}$ and $f\circ i:X\to K\setminus\{y\cup Ay\}$ is equivariant. Applying the first statement we obtain that $\hind\!_2 (K\setminus\{y\cup Ay\}) \ge d$.
On the other hand $K\setminus\{y\cup Ay\}$ is an open manifold, hence $H^j(K\setminus\{y\cup Ay\};\mathbb Z_2)=0$ for $j\ge d$, and from property~7 of index we
obtain $\hind\!_2 (K\setminus\{y\cup Ay\}) < d$.
\end{proof}

Similar argument proves the following theorem.
\begin{theorem} Let $X$ be a free $\mathbb Z_2$-space, $i:X\subset Z$. Let $K$ be a free $\mathbb Z_2$-space with involution  $A:K\to K$ and $f:Z\to K$ is a map equivariant on $X$.
Assume that $\hind_{\mathbb Z} X=d-1$ and that the inclusion $i:X\subset Z$ induces trivial homomorphism $0=i^*:H^{d-1}(Z;\mathbb Z)\to H^{d-1}(X;\mathbb Z)$.
Then $\hind_{\mathbb Z} K \ge d$.

If in addition $K$ is a connected closed topological $d$-dimensional manifold or a pseudomanifold then for any $y\in K$ at least one of the sets $f^{-1}(y)$, $f^{-1}(Ay)$ is nonempty.
\end{theorem}

Let $Y$ be a free $\mathbb Z_2$-space. The action of $\mathbb Z_2$ on $Y$ is naturally extended to the action on the cone $CY$ with a unique fixed point, the vertex $c\in CY$ of the cone.

\begin{theorem} Let $X$ and $Y$ be free $\mathbb Z_2$-spaces such that $\hind\!_{2}(X) = k=\hind\!_{2}(Y)$.
Assume that $X$ is a closed subset of a paracompact space $Z$ and that the inclusion $i:X\to Z$ induces the trivial homomorphism $i^*:H^k(Z;\mathbb Z_2)\to H^k(X;\mathbb Z_2)$.
Then for any continuous $f:Z\to CY$ which is equivariant on $X$ the vertex $c\in CY$ of the cone $CY$ is in the image of $f$.
\end{theorem}
\begin{proof} An obvious map $p:CY\stm c\to Y$ is equivariant. If $f(Z)\subset CY\stm c$ then
$f\circ i:X\to Y$ is equivariant and induces a trivial homomorphism
of cohomology groups in dimension $k$. This contradicts with the property~8.
\end{proof}
Similar assertion holds for $\hind_{\mathbb Z}(\,\cdot\,)$.

\begin{theorem} Let $X$ be a free $\mathbb Z_2$-space, $i:X\subset Z$. Let $M$ be a connected closed topological $d$-dimensional manifold (or a pseudomanifold)
with a non--free involution  $A:M\to M$. Let $F$ be the fixed point set of $A$ and assume that $f:Z\to M$ is a map equivariant on $X$.
If \, $\hind_{\mathbb Z} X=d-1$ and the inclusion $i:X\subset Z$ induces trivial homomorphism $0=i^*:H^{d-1}(Z;\mathbb Z)\to H^{d-1}(X;\mathbb Z)$
then $f^{-1}F\not=\emptyset$.
\end{theorem}
\begin{proof} By contradiction, using Property 8 and the inequality  $\hind_{\mathbb Z} (M\setminus F)<d$.
\end{proof}
Similar assertion holds for $\hind\!_{2}(\,\cdot\,)$.
\medskip

In a case when  $X\subset Z$ and $X$ is a free $\mathbb Z_2$-space we can double the space $Z$ and obtain
the space $D(X,Z)$ with a free involution such that $D(X,Z)$ is the union of two subspaces homeomorphic to $Z$ and
$X\subset D(X,Z)$ is an invariant subspace, see~\cite{MusS}. The above results can be obtained using this space $D(X,Z)$.

It is convenient to explain the construction in more general situation when $X$ is a free $G$-space where $G$ is finite.
Let $X$ be a compact subspace of a compact space $Z$ (no action of $G$ on $Z\setminus X$ assumed). Denote by $i:X\to Z$ the inclusion.
Then there exists a compact space $D(X,Z)$, {\it camomile}, with a free $G$-action and the inclusion $j:Z\to D(X,Z)$ such that
$j\circ i:X\to D(X,Z)$ is equivariant and $D(X,Z)$ is a disjoint union
$$D(X,Z)=j(X)\bigcup \bigcup\limits_{g\in G}gj(Z\setminus X).$$
We explain the construction for simplicial complexes.

In case when $Z$ is a simplicial complex and $X$ its subcomplex (or more general for $CW$-complexes) the space $D(X,Z)$ is obtained by the well known
procedure of equivariant attachments of cells. Let $\varphi:\mathbb S^{k-1}\to Y$ be any continuous map to a $G$-space $Y$. Consider $|G|$ maps
$g\circ \varphi:\mathbb S^{k-1}\to Y$, $g\in G$. Attach $k$-dimensional cells to $Y$ using this maps. Since a ball is a cone over its boundary,
we can easily extend the $G$-action on the space $Y$ with attached cells. Consider the disjoint union of $X$ with
$|G|$ copies of the set of 0-dimensional simplices (0-cells) in $Z\setminus X$ and associate with each copy an element of $G$.
A free action on the defined space is easily defined. Attach further 1-cells to the union of $X$ and a set associated with the unit of $G$
and make equivariant attachment of 1-cells as explained above. Successive equivariant attachment of cells of dimensions 2, 3 etc. gives us
the camomile $D(X,Z)$. In simplicial case we can make all maps simplicial, so $D(X,Z)$ will be a simplicial complex with simplicial action of $G$.

\begin{remark}{\rm Let $Y$ be a free $\mathbb Z_2$-space and $\tind Y=d$. Then $Y$ can be represented as $Y=D(X,Z)$ with $X$ such that $\tind X=d-1$.
Indeed there exists equivariant map $f:Y\to \mathbb S^d$ and we denote by $Z$ the preimage of upper hemisphere and $X$ the preimage of the equator.}
\end{remark}

\begin{prop} Assume that $\hind\!_2 X\geq d-1$ and $0=i^*:H^{d-1}(Z;\mathbb Z_2)\to H^{d-1}(X;\mathbb Z_2)$.
Then $\hind_2 D(X,Z)\geq d$.
\end{prop}
\begin{prop} Assume that $\hind_{\mathbb Z} X\geq d-1$ and $0=i^*:H^{d-1}(Z;\mathbb Z)\to H^{d-1}(X;\mathbb Z)$.
Then $\hind_{\mathbb Z} D(X,Z)\geq d$.
\end{prop}

Consider $X=\mathbb S^{d-1}$ with the antipodal (or any other free) involution and let $i:X\to Z$ be an inclusion such that
$0=i^*:H^{d-1}(Z;\mathbb Z_2)\to H^{d-1}(\mathbb S^{d-1};\mathbb Z_2)$. It follows from the above proposition that $\hind\!_2 D(X,Z)\geq d$.

In particular consider a closed connected topological $d$-dimensional manifold $M$. Denote by $Z$ the manifold $M$ with deleted open ball and by
$X$ the boundary of this ball which is homeomorphic to $\Bbb S^{d-1}$. Then $D(X,Z)$ is the connected sum $M\# M$, so
$M\#M=M_1\cup M_2$ where $M_1$ is mapped to $M_2$ and $M_2$ to $M_1$ by the involution on $M\#M$ and
$M_1\cap M_2=\mathbb S^{d-1}$. Since $M$ is closed, the inclusion $i:\mathbb S^{d-1}\to M_1$ induces the trivial homomorphism of $(d-1)$-dimensional
cohomology and we obtain that $\hind\!_2 M\#M=d$.

For a pseudomanifold or more general for a connected $d$-dimensional simplicial complex $K$ we can define $K\# K$ by deleting a small open ball
from the interior of some chosen $d$-simplex (the construction depends on this choice).

Thus we have the following statement.

\begin{corr} Let $M$ be a closed connected topological manifold or a closed pseudomanifold. Then $\hind\!_2 M\#M=\dim M$.
\end{corr}




The following result is similar to the results from Section 2.

\begin{theorem} \label{t42}  Let $Z$ be a finite simplicial complex and $X$ is its a subcomplex which is a free $\mathbb Z_2$-space with a free involution $A$. Assume that at least one of the following conditions holds:

a) $\hind\!_{2} X\geq d-1$ and $0=i^*:H^{d-1}(Z;\mathbb Z_2)\to H^{d-1}(X;\mathbb Z_2)$,

b) $\hind_{\mathbb Z} X\geq d-1$ and $0=i^*:H^{d-1}(Z;\mathbb Z)\to H^{d-1}(X;\mathbb Z)$.

Then the following statements hold.
\begin{enumerate}
\item For any covering of $Z$ by  a family of $2n$ closed (respectively, of  $2n$ open) sets $\{C_1,C_{-1},\ldots,C_n,C_{-n}\}$, where $A(C_i)\cap X=C_{-i}\cap X$ for all $i=1,\ldots,n$, there is $k$ such that $C_k$ and $C_{-k}$ have a common intersection point. 	
	
\item For  any triangulation $\cal T$ of  $Z$ that is equivariant on $X$ and a $\Pi_m$-labeling of the vertex set $V({\cal T})$ of triangulation $\cal T$
which is antipodal on $X$ there exists a complementary edge in $\cal T$.

\item  For  any triangulation $\cal T$ of  $Z$ that is antipodal on $X$ and for any  labeling
$L:V({\cal T})\to \Pi_m$ without   complementary edges  which is antipodal on $X$ there is an $n$-simplex in $\cal T$ with labels in the form $\{k_0,-k_1,k_2,\ldots,(-1)^nk_n\}$, where  $1\le |k_0|<|k_1|<\ldots<|k_n|\le m$ and all $k_i$ are of the same sign.

\item  For any  centrally symmetric set of\, $2m$ points $\{p_1,p_{-1},\ldots,p_m,p_{-m}\},\, p_{-k}=-p_k,$ in ${\Bbb R}^n$, for any  triangulation $\cal T$ of $Z$ that is antipodal on $X$ , and for any labeling  $L:V({\cal T})\to \Pi_m$ that is antipodal on the boundary, there exists a $k$-simplex in $\cal T$ with labels $\ell_0,\ldots,\ell_ k$ such that the convex hull of points  $\{p_{\ell_0},\ldots,p_{\ell_k}\}\subset{\Bbb R}^n$ contains the origin $0\in {\Bbb R}^n$.

\item For any set-valued function $F:Z\to 2^{{\Bbb R}^n}$  with a closed graph  such that for all $z\in Z$ the set $F(z)$ is non-empty compact and convex in ${\Bbb R}^n$, and for points $x\in X$ the set $F(x)$ contains $y$ with $(-y)\in F(A(x))$, there is $x_0\in X$ such that  $F(x_0)$ covers the origin $0\in{\Bbb R}^n$.

\end{enumerate}
\end{theorem}	

\begin{remark} {\rm Actually, statement 5 in this form cannot be directly derived from Theorem~\ref{KBUT}. However, it can be proved by the same argument as Theorem \ref{KBUT}. On the other hand we can apply homological methods in multivalued function theory~\cite{Gor}. For example assume condition a) and acyclicity
 (for cohomology with coefficients in $\mathbb Z_2$) of spaces $F(z)$, $z\in Z$ and assume that $F(Ax)=-F(x)$ for all $x\in X$. Denote by $\Gamma\subset Z\times \mathbb R^n$ the graph of $F$ and by $\Gamma_X\subset X\times \mathbb R^n$ the graph of $F|_X$. We have a free involution $A$ on $\Gamma_X$ defined as
$A(x,y)=(Ax,-y)$, $x\in X$ and $y\in F(x)$, and an equivariant map $g:\Gamma_X\to \mathbb R^n$, $g(x,y)=y$.
From Vietoris--Begle theorem it follows that $\hind\!_2 \Gamma_X=\hind\!_2 X$. Hence for some pair $(x_0,y_0)$ we have $g(x_0,y_0)=0$, i.e. $y_0=0$, and therefore $0\in F(x_0)$.
}
\end{remark}

\section{BUT--manifolds}
In this section a manifold means either
connected topological manifold, or closed pseudomanifold.
The involutions on pseudomanifolds are assumed to be simplicial.

First recall the notion of the degree of a map between manifolds.

Let $M$, $N$ be such manifolds of dimension $d$ and without boundary. Consider a map $h:M\to N$.
Then $\deg_2 h$, mod 2 degree of $h$, and $\deg h$ in a case when $M$, $N$ are orientable, are defined. We have
$H^d(M;\mathbb Z_2)\cong \mathbb Z_2\cong H^d(N;\mathbb Z_2)$ and we put $\deg_2 h=1$  (respectively $\deg_2 h=0$) if
$h^*\not=0$ (respectively $h^*=0$) where $h^*:H^d(N;\mathbb Z_2)\to H^d(M;\mathbb Z_2)$ is a homomorphism induced by $h$.
If $M$, $N$ are orientable then $H^d(M;\mathbb Z)\cong \mathbb Z\cong H^d(N;\mathbb Z)$ and $\deg h\in \mathbb Z$ is defined
from the equation $h^*\alpha_N=\deg h\cdot\alpha_M$ where $\alpha_M, \alpha_N$ are generators (depending on orientations of manifolds).
Note that the property of the degree to be odd or even does not depend on a choice of orientations and $\deg_2 h=\deg h\,\,(\mbox{mod }\, 2)$.

Recall that if $M$, $N$ and involutions on them are smooth then $\deg_2 h$ (respectively $\deg h$) equals the number modulo 2
(respectively algebraic number) of points in the preimage of any regular value of $h$.

In the case when $M$, $N$ are pseudomanifolds and $h:M\to N$ is simplicial the degree of $h$ can be defined as follows: $\deg_2 h$ equals the number modulo 2
of simplices in the preimage of
any $d$-dimensional simplex in $N$ (simplices with falling dimension under the map $h$ are not taken into account).
In the case when $M$, $N$ are orientable fix the orientation of these manifolds and a simplex $\sigma \subset N$.
Then $\deg h$ equals the difference between the number of simplices
in the preimage of $\sigma $ for which $h$ preserves orientations and number of
simplices for which $h$ reverses orientations.

\medskip

It is well known that Borsuk--Ulam theorem is equivalent to the following assertion:

{\it Every continuous odd mapping ${\Bbb S}^d\to {\Bbb S}^d$  has odd degree.}

Here odd or antipodal means equivariant with respect to the antipodal involution on ${\Bbb S}^d$.

\medskip

\begin{theorem}\label{odddeg} Let $M$, $N$ be $d$-dimensional $\but_d$-manifolds, i.e. $\tind M=d=\tind M$. If
there exists an equivariant map $f:M\to N$ then $\deg_2 f=1$.

If $M$, $N$ are orientable then \,$\deg f$ is odd.
\end{theorem}
\begin{proof} Since $\tind M=d=\tind M$, we have $\hind\!_2 M=d=\hind\!_2 M$, and the statement follows
directly from the definition of degrees and property~8 of index $\hind\!_2 (\,\cdot\,)$.
\end{proof}

\begin{remark}{\rm
Let $M$ be a $d$-dimensional manifold with a free involution. Then there exists an equivariant map $h:M\to \mathbb S^d$ and Theorem~\ref{odddeg}
implies that if $M$ is
$\but_d$ then $\deg_2 h=1$. If moreover $M$ is orientable then \,$\deg h$ is odd. Actually this statement is equivalent to the Theorem~\ref{odddeg}. In fact
take any equivariant map $g:N\to \mathbb S^d$. Applying Theorem~\ref{odddeg} to the equivariant map $h=g\circ f:M\to \mathbb S^d$
we see that  $\deg_2 g\circ f=1$. Hence $\deg_2 f=1$. Now we can give another proof of Theorem~\ref{odddeg} using the following
Conner--Floyd type generalization~\cite{CF} of the Borsuk--Ulam theorem:

{\it Let $X$ be a space with a free involution $A$ and such that $\hind\!_2 X\ge d$. Let $R$ be a $d$-dimensional
topological manifold and $\varphi:X\to R$ a map such that
$\varphi^*$, the homomorphism of cohomology with coefficients in $\mathbb Z_2$, is trivial in all positive dimensions. Then there exists $x\in X$ such that
$\varphi(x)=\varphi(Ax)$.}

This statement shows that the supposition $\deg_2 h=0$ leads to a contradiction ($h:M\to \mathbb S^d$ cannot be equivariant).
}
\end{remark}

\medskip
From theorems  \ref{tEQUIV}, \ref{odddeg} we obtain:
\begin{theorem} Let $h:M\to \mathbb S^d$ be an equivariant map of
a $d$-dimensional manifold $M$. Then $M$ is $\but_d$, i.e. $\tind M=d$, if and only if \,$\deg_2 h=1$.

If moreover $M$ is orientable then it is $\but_d$ if and only if \,$\deg h$ is odd.
\end{theorem}

For a smooth $d$--dimensional manifold $M$ (with a smooth free involution) and a smooth antipodal map $h:M \to {\Bbb R}^d$ we say that $h$ is
transversal to zero if $0\in {\Bbb R}^d$ is a regular value of $h$. In this case $h^{-1}(0)$ consists of finite number of points. To construct
such a map $h$ choose any smooth equivariant map $f:M\to \mathbb S^d$ such that North and South poles of $\mathbb S^d\subset \mathbb R^{d+1}$
are regular values of $f$. Note that $0\in \mathbb R^{d}$ is a regular value of the projection $\pi:\mathbb S^d\to\mathbb R^{d}$ where
$\pi(x_1,\dots,x_{d+1})=(x_1,\dots,x_{d})$,\, $x_i$ are coordinates in $\mathbb R^{d+1}$. Hence $0\in \mathbb R^{d}$ is a regular value of
the composition map $\pi\circ f:M\to\mathbb R^{d}$. We see that if $\deg\!_2 f=1$ then the preimage of the North (South) pole under $f$ consists of $2k+1$ points and therefore $|h^{-1}(0)|=4k+2$.

In a case $M$ is a pseudomanifold we say that $h:M \to {\Bbb R}^d$ is
transversal to zero if the map $h$ is linear on each simplex and $h^{-1}(0)$ consists only from points lying in the interiors of $d$-simplexes.
To show the existence of such a map we need only to construct for any antipodal triangulation of $\mathbb S^d$ a transversal to zero map $S^d\to\mathbb R^{d}$.

Consider any triangulation of the sphere $\mathbb S^d$ invariant under the antipodal involution, so consider the sphere as a boundary of convex centrally symmetric polytope. Take any interior point $x$ of an arbitrary $d$-simplex.
Orthogonal projection of a sphere $\mathbb S^d\subset\mathbb R^{d+1}$ on the linear $d$-dimensional subspace orthogonal to a line through points $x$ and $-x$ is an equivariant transversal to zero map $S^d\to\mathbb R^{d}$, and the preimage of zero for this map consists of exactly two points $x$ and $-x$. It follows
that $|h^{-1}(0)|=4k+2$ for a transversal to zero if the map $h:M \to {\Bbb R}^d$.

Thus we have proved
\begin{theorem} Let $M$ be either a connected closed smooth $d$--dimensional manifold with a free smooth involution
 or closed pseudomanifold with  a free  simplicial involution.
Then $M$ is a $\but_d$-space if and only if $M$ admits an antipodal transversal to zero mapping $h:M \to {\Bbb R}^d$ with the cardinality of zeros set $|h^{-1}(0)|=4k+2$, where $k\in \{0, 1, \dots\}$.
\end{theorem}

Let $X$ and $K$ be spaces with free involutions and $X\subset Z$. Any map $f:Z\to K$ which is equivariant on $X$ is uniquely extended
in the obvious way to the equivariant map $\widetilde f:D(X,Z)\to K$.

In particular let $X_i$, $i=1,\,2$, be spaces with free involutions. Assume that $X_i\subset Z_i$, $i=1,\,2$, and $f:(Z_1, X_1)\to (Z_2, X_2)$ is
a map equivariant on $X_1$. Then $f$ induces the equivariant map $Df:D(X_1,Z_1)\to D(X_1,Z_2)$.

\medskip

\begin{theorem}
Let $M$ be a $d$--dimensional pseudomanifold with boundary $\partial M$ which is a $(d-1)$--dimensional pseudomanifold
with a free simplicial involution.
Suppose  $\partial M$ is a BUT--manifold, $N$ is a closed $d$--dimensional pseudomanifold with a free simplicial involution $A:N\to N$ and $f:M\to N$ a simplicial map equivariant on $\partial M$. For any $d$-simplex $\sigma\subset N$ the number of $d$-simplexes in $M$ mapped either on
$\sigma$ or on $A\sigma$ is odd.
\end{theorem}
\begin{proof} Since $\partial M$ is a BUT--manifold, we have $\hind\!_2 \partial M=d-1$. From the assumptions concerning $M$ and $\partial M$ it follows
that the inclusion $\partial M\to M$ induces the trivial homomorphism of $(d-1)$-dimensional cohomology groups with $\mathbb Z_2$ coefficients. Hence
the equivariant embedding $\partial M\to D(\partial M,M)$ also induces trivial homomorphism of $(d-1)$-dimensional cohomology groups. From
property~8 of index $\hind\!_2(\,\cdot\,)$ it follows that $\hind\!_2 D(\partial M,M)=d$. Note also that $D(\partial M,M)$ is a closed $d$-dimensional pseudomanifold.

Consider the extension $\widetilde f:D(\partial M,M)\to N$ of $f:M\to N$. Since $\widetilde f$ is equivariant we see that $\hind\!_2 N=d$.
From property~8 of index it follows that $\deg\!_2 \widetilde f=1$. Hence for any $d$-dimensional simplex $\sigma\subset N$ either
$f^{-1}(\sigma)$ or $f^{-1}(A\sigma)$ contains an odd number of simplices (not degenerating under $f$).
\end{proof}

As a partial case of this theorem we obtain for Tacker labeling the following Shashkin type result (see   \cite{MusS,ShashkinT,Shashkin99}) in which $\Pi_n$ denotes the set of labels $\{+1,-1,+2,-2,\ldots, +n,-n\}$.
\begin{theorem} Let $M$ be a $d$--dimensional pseudomanifold with boundary $\partial M$ which is a $(d-1)$--dimensional pseudomanifold with a free simplicial involution. Suppose that $\partial M$ is a BUT--manifold (i.e. $\tind \partial M=d-1$) and there is given a
$\{\pm 1,\dots,\pm(d+1)\}$-labeling
that is antipodal on the boundary $\partial M$ and has no complementary edges. Then for any set of  labels $\Lambda:=\{\ell_1,\ell_2,\ldots,\ell_{d+1}\}\subset \Pi_{d+1}$ with $|\ell_i|=i$ for all $i$, the number of $d$-simplexes that are labeled by one of $\Lambda$ or $(-\Lambda)$  is odd.
\end{theorem}
\begin{proof} In place of $N$ we take $d$-dimensional octahedral sphere. The labeling provides a map of $M$ to this sphere and we can apply
the previous theorem.
\end{proof}

Similar argument provides the following theorem, a version of Shashkin's lemma for manifolds without boundaries.

\begin{corr} Let $M$ be a $d$--dimensional BUT-pseudomanifold (i.e. $\tind M=d$). Assume that there is given an equivariant
$\{\pm 1,\dots,\pm(d+1)\}$-labeling of the vertex set of $M$ and this labelling does not have complementary edges.
Then for any set of labels $\Lambda:=\{\ell_1,\ell_2,\ldots,\ell_{d+1}\}\subset\Pi_{d+1}$ with $|\ell_i|=i$ for all $i$, the number of $d$--simplexes in $T$ that are labelled by $\Lambda$  is odd.
\end{corr}

\medskip

We finish this section with several more facts about $\but$-manifolds that are smooth or $PL$.

\medskip

Let us recall several  facts about ${\Bbb Z}_2$-cobordisms.
 We write ${\mathfrak N}_n$ for the group of unoriented cobordism classes of smooth $n$-dimensional manifolds.    Thom's cobordism theorem says that
the graded ring of cobordism classes  ${\mathfrak N}_*$ is
${\Bbb Z}_2[x_2, x_4, x_5, x_6, x_8, x_9,\ldots]$ with one generator $x_k$ in each degree $k$ not of the form $2^i - 1$.
Note that $x_{2k}=[{\Bbb R}{\Bbb P}^{2k}]$.

Let  ${\mathfrak N}_*({\Bbb Z}_2)$ denote the unoriented cobordism group of  free involutions. Then ${\mathfrak N}_*({\Bbb Z}_2)$ is a free ${\mathfrak N}_*$-module with basis $[{\Bbb S}^n,A]$, $n\ge0$, where $[{\Bbb S}^n,A]$ is the cobordism class of the antipodal involution on the $n$-sphere \cite[Theorem 23.2]{CF}. Thus,  each ${\Bbb Z}_2$-manifold can be uniquely represented in ${\mathfrak N}_n({\Bbb Z}_2)$  in the form:
$$
[M,A]=\sum\limits_{k=0}^n {[V^k][{\Bbb S}^{n-k},A]}.
$$

Recall also that a continuous mapping $h:M \to {\Bbb R}^n$ is called {\it transversal to zero} if there is an open set $U$ in ${\Bbb R}^n$ such that $U$ contains $0$, $U$ is homeomorphic to the open $n$-ball and $h^{-1}(U)$ consists of a finite number open sets in $M$ that are homeomorphic to open $n$-balls.

The following theorem is proved in \cite[Theorem 2]{Mus}.

\begin{theorem} \label{thmBUT} Let $M$ be a connected closed smooth $n$-dimensional manifold with a free smooth involution $A$. Then the following statements are equivalent:

\begin{enumerate}
\item $(M,A)$ is a $\but_n$-space.

\item $w_1^n(M/A)\ne0$ in $H^n(M/A;{\Bbb Z}_2)$, where $w_1(M/A)$ is 1-dimensional Stiefel--Whitney class of $\mathbb S^0$-bundle $M\to M/A$.

\item $M$ admits an antipodal continuous transversal to zeros mapping $h:M \to {\Bbb R}^n$ with the cardinality of zeros set $|h^{-1}(0)|=4k+2$, where $k\in {\Bbb Z}$.

\item For any antipodal map $f:M\to \mathbb S^n$ we have $\deg_2 f=1$.

\item $[M^n,A]=[{\Bbb S}^n,A]+[V^1][{\Bbb S}^{n-1},A]+\ldots+[V^n][{\Bbb S}^0,A]$ in ${\mathfrak N}_n({\Bbb Z}_2)$.
\end{enumerate}
\end{theorem}
Equivalence of statements 1,2,3,4 was discussed above. Bordism invariance of Stiefel--Whitney numbers~\cite{CF} shows the equivalence of statements~2  and~5.

\medskip

In \cite[Theorem 3]{Mus} we considered a sufficient condition when an antipodal continuous mapping $f:M\to{\Bbb R}^n$ has non-empty zeros set $Z_f:={f^{-1}(0)}$. However, this condition is not necessary.

The following theorem is a version of Theorem \ref{tBUTS}, statement 5, for $BUT$--manifolds.

\begin{theorem} \label{strat} Let $M$ be a connected  $m$-dimensional, $m\ge n$, closed PL manifold with a free simplicial involution $A$. Then the following two statements are equivalent:

\begin{enumerate}
	
\item $(M,A)$ is a $\but_n$-space.

\item There are submanifolds  $M_1\subset M_2\subset\ldots\subset M_{n-1}\subset M_n=M$ such that for all $i$  $M_i$ is a closed smooth manifold of dimension $m-n+i$, $A(M_i)=M_i$, $M_i$ is a $\but_{i}$--manifold, and $M_{i+1}\setminus M_i$ consists of two connected components $C_1, C_2$ with $A(C_2)=C_1$.

\end{enumerate}	
\end{theorem}
\begin{proof} Using the same arguments as in \cite[Section 2]{Mus} (see Lemmas 2 and 3)  we can prove that there are subspaces ${\Bbb R}^1\subset{\Bbb R}^2\subset\ldots\subset{\Bbb R}^{n-1}\subset{\Bbb R}^n$ and a regular antipodal mapping $h:M\to{\Bbb R}^n$ such that $M_i:=h^{-1}({\Bbb R}^i)$ is a manifold of dimension $m-n+i$ for all i. It is easy to see that these submanifolds satisfy all other properties.
On the other hand, it is obvious that statement 2 implies statement 1.
\end{proof}

\begin{remark} {\rm In smooth case such sequences can be obtained by the procedure similar to the definition of Smith's homomorphism in cobordism, see \cite{CF}. }
\end{remark}

This theorem immediately yields a classification of two-dimensional $\but_2$--manifolds. Namely, this class consists of orientable  two-dimensional  manifolds $M^2_g$ with even genus $g$ and non-orientable manifolds $P^2_m$ with even $m$, where $m$ is the number of  M\"obius bands.

\medskip

\noindent{\bf Conjecture.} All $M_i$ in Theorem \ref{strat} can be chosen connected.

\medskip

If this conjecture is correct then statement 3 in Theorem \ref{thmBUT} can be substituted by\\
{\it ``$M$ admits an antipodal continuous transversal to zeros mapping $h:M^n \to {\Bbb R}^n$ such that $h^{-1}(0)$ consists of two points.''}

\medskip

In forthcoming paper \cite{MusVo} we will discuss Tucker type results for $G$-spaces.

\medskip

\medskip


 \medskip

O. R. Musin, 
University of Texas Rio Grande Valley, School of Mathematical and Statistical Sciences,  One West University Boulevard,  Brownsville, TX, 78520.\\
{\it E-mail address:} oleg.musin@utrgv.edu

 \medskip

A. Yu. Volovikov,  Department of Higher Mathematics, MSTU MIREA, Vernadskogo av., 78, Moscow, 119454, Russia\\
{\it E-mail address:} $\mbox{a\_volov@list.ru}$

\end{document}